\documentclass[12pt, a4paper]{article}

\usepackage{a4wide}
\usepackage{subfigure}
\usepackage[font=small, labelsep=period, figurename=Fig.]{caption}
\usepackage{graphicx}
\usepackage{amsthm}
\usepackage{amsmath}

\theoremstyle{plain}
\newtheorem{theorem}{Theorem}
\newtheorem{lemma}{Lemma}

\allowdisplaybreaks[3]

\begin{document}

\title{On Crosspatch Knight's Tours}
\author{Nikolai Ivanov Beluhov}
\date{}

\maketitle

\begin{abstract}

A knight's tour is often represented as a broken line connecting the centers of successively visited squares. We say that two knight moves form a cross if the midpoints of their respective segments coincide. We show that no knight tour exists on a rectangular board in which every move is part of a cross. We also establish the general structure of pseudotours with this property.   
	
\end{abstract}

\section{Introduction}

A knight's tour is often depicted as a broken line connecting the centers of successively visited squares, each knight move represented by a segment in the line.

In this setting, it is natural to ask whether a tour exists subject to various local restrictions of a purely geometric nature. For instance, in \cite{j76}, G. P. Jelliss has shown that a closed knight tour which does not contain a right angle (i.e., two successive moves at right angles) does not exist on an $8 \times 8$ board; in \cite{j85}, that one exists on a $10 \times 10$ board and all larger boards; and in \cite{b06}, the present author has shown that a closed knight tour which does not contain an obtuse angle does not exist on an $8 \times 8$ board but exists on a $10 \times 10$ board.

Following the terminology of \cite{j84}, we say that two knight moves form a \emph{central cross}, or simply \emph{cross}, if their midpoints coincide. (Fig.\ \ref{f1}) We say that a knight graph is a \emph{crosspatch} graph if every knight move in it is part of a cross.

\begin{figure}[h!]
\centering
\includegraphics[scale=1.0]{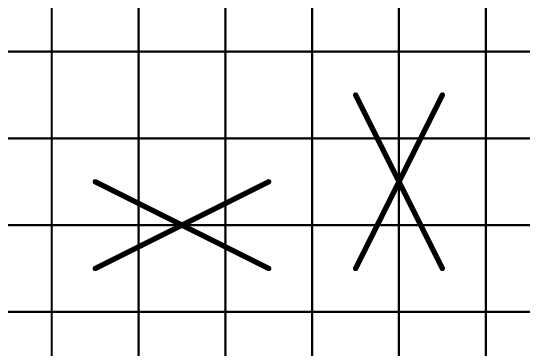}
\caption{}
\label{f1}
\end{figure}

In order to study crosspatch tours, it proves rather helpful to examine crosspatch \emph{pseudotours} as well. A knight pseudotour, as defined by Jelliss, is a knight graph in which every vertex is of degree two. Obviously, a pseudotour is made up of a number of independent cycles.

In \cite{j84}, Jelliss conjectures that the only connected chessboard which admits a crosspatch tour is the eight-square ring obtained by removing the central square from a $3 \times 3$ board, and enumerates all crosspatch pseudotours on an $8 \times 8$ board.

In the present paper, we give a complete description of the structure of crosspatch pseudotours on rectangular boards. From this description, we derive the non-existence of crosspatch tours on boards of this type.

The statement of Theorem \ref{t1} below was suspected by the author in 2007, but a proof was elusive until October 2012.

\section{Crosspatch Pseudotours}

Given an $m \times n$ rectangular chessboard $B$, let $G$ be a crosspatch pseudotour on $B$ whose vertices are the centers of $B$'s squares and whose edges are undirected knight moves.

\begin{figure}[ht]
\hspace{0pt}\hfill\begin{minipage}{0.5\linewidth}
\centering
\includegraphics[scale=0.75]{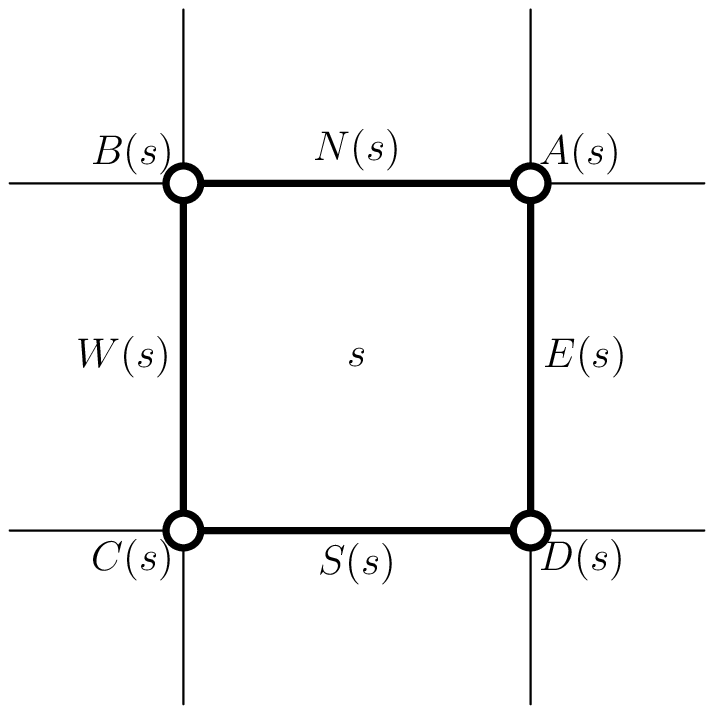}
\caption{}
\label{f2}
\end{minipage}\hfill\begin{minipage}{0.5\linewidth}
\centering
\includegraphics[scale=0.875]{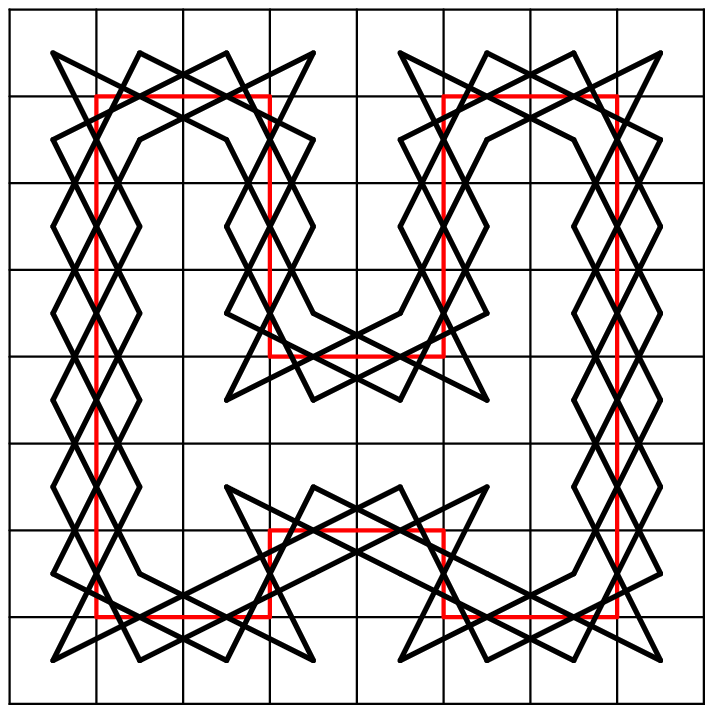}
\caption{}
\label{f3}
\end{minipage}\hfill\hspace{0pt}
\end{figure}

We begin by introducing some useful notation.

We label the columns of $B$ with the numbers from $1$ to $m$, the rows of $B$ with the numbers from $1$ to $n$, and every square in $B$ with the ordered pair $(i, j)$ of the labels of its column and row.

If $s$ is a square in $B$, then let $N(s)$, $E(s)$, $S(s)$, and $W(s)$ be its top, right, bottom, and left edge, respectively, and let $A(s)$, $B(s)$, $C(s)$, and $D(s)$ be its top-right, top-left, bottom-left, and bottom-right vertex, respectively, as in Fig.\ \ref{f2}.

If $s \equiv (i, j)$, then we will also write $N(i, j)$ for $N(s)$, $E(i, j)$ for $E(s)$, etc. In this manner, $E(i, j) \equiv W(i + 1, j)$, $N(i, j) \equiv S(i, j + 1)$, and $A(i, j) \equiv B(i + 1, j) \equiv C(i + 1, j + 1) \equiv D(i, j + 1)$ for all $i$, $j$ such that the corresponding squares exist.

In order to distinguish between the vertices and edges of $G$ and the vertices and edges of the squares in $B$, we will refer to the former as \emph{graph} vertices and edges, and to the latter as \emph{board} vertices and edges.

Finally, label each board vertex with the ordered pair of the labels of the column on its left (0 for the leftmost board vertices) and the row just below it (0 for the bottommost board vertices), and, for every board vertex $v$, let $N(v)$, $E(v)$, $S(v)$, and $W(v)$ be the board edges adjacent to it and pointing up, right, down, and left from it, respectively. In this manner, $N(v) \equiv W(C^{-1}(v)) \equiv E(D^{-1}(v))$, $E(v) \equiv N(B^{-1}(v)) \equiv S(C^{-1}(v))$, etc., for all board vertices $v$ such that the corresponding squares $A^{-1}(v)$, $B^{-1}(v)$, etc., exist.

For every cross formed by two graph edges in $G$, colour in red the board edge whose midpoint coincides with the center of the cross, and let $H$ be the graph whose vertices are all board vertices and whose edges are all red board edges. Fig.\ \ref{f3} depicts a typical crosspatch pseudotour $G$ together with the associated $H$.

\begin{lemma}
\label{l1}

Every board vertex in $H$ has an even degree.

\end{lemma}

Note that this statement does not hold on infinite, cylinder, or toroidal boards, as seen in Fig.\ \ref{f6}. This is a strong indication that a proof must necessarily make use of some sort of non-local argument involving the board's boundary.

\begin{figure}[ht]
\centering
\includegraphics[scale=0.875]{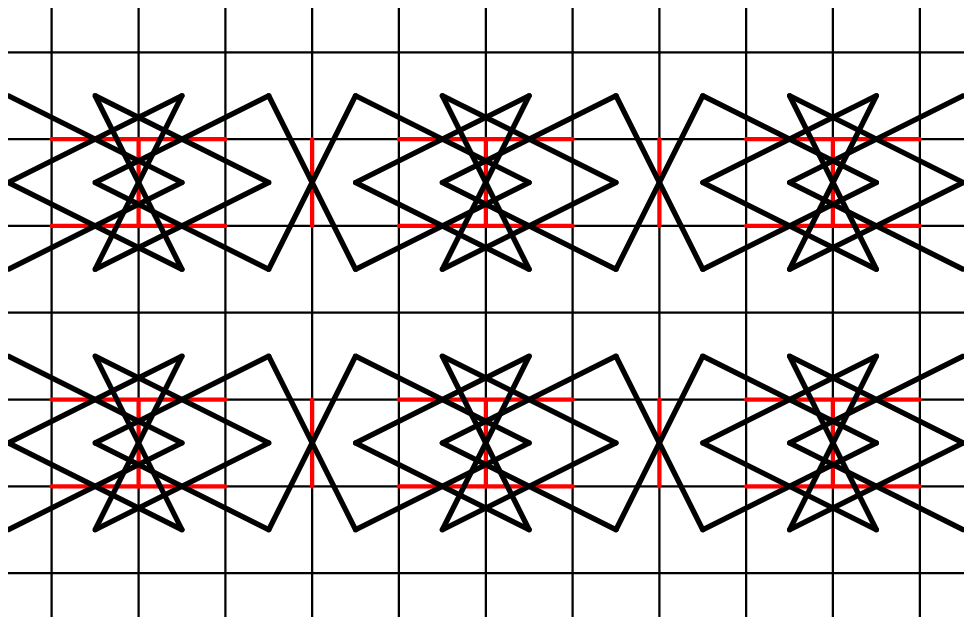}
\caption{}
\label{f6}
\end{figure}

\begin{proof}

For every board edge $e$, let $\alpha(e)$ be equal to 1 if $e$ is red, and 0 otherwise. Notice that, if $e$ is adjacent to a border board vertex, then $\alpha(e) = 0$.

For every square $s$, the midpoints of the two graph edges adjacent to $s$ coincide with the midpoints of exactly two of the board edges $E(A(s))$, $N(A(s))$, $N(B(s))$, $W(B(s))$, $W(C(s))$, $S(C(s))$, $S(D(s))$, and $E(D(s))$. (Some of these board edges will not exist for some $s$.) Conversely, for every red board edge among those, one of the graph edges in its corresponding cross is adjacent to $s$. It follows, then, that, for all $s$, \[
\begin{split}
2 =\; &\alpha(E(A(s))) + \alpha(N(A(s))) + \alpha(N(B(s))) + \alpha(W(B(s))) +\\ 
  +\; &\alpha(W(C(s))) + \alpha(S(C(s))) + \alpha(S(D(s))) + \alpha(E(D(s))).
\end{split}
\]

Let $u \equiv (a, b)$ be an arbitrary board vertex. Then
\begin{align*}
2ab\; &= \sum\limits_{\substack{1 \le i \le a\\ 1 \le j \le b\\ s \equiv (i, j)}} (\alpha(E(A(s))) + \ldots + \alpha(E(D(s)))) = \\
    &=\; 4\sum\limits_{\substack{1 \le i \le (a - 1)\\ 1 \le j \le (b - 2)\\ v \equiv (i, j)}} \alpha(N(v))\;+\; 
      4\sum\limits_{\substack{1 \le i \le (a - 2)\\ 1 \le j \le (b - 1)\\ v \equiv (i, j)}} \alpha(E(v))\;+ \\ 
    &+\; 2\sum\limits_{\substack{i = a\\ 1 \le j \le (b - 1)\\ v \equiv (i, j)}} (\alpha(W(v)) + \alpha(E(v)))\;+\;
      2\sum\limits_{\substack{i = a\\ 1 \le j \le (b - 2)\\ v \equiv (i, j)}} \alpha(N(v))\;+ \\ 
    &+\; 2\sum\limits_{\substack{1 \le i \le (a - 1)\\ j = b\\ v \equiv (i, j)}} (\alpha(S(v)) + \alpha(N(v)))\;+\;
      2\sum\limits_{\substack{1 \le i \le (a - 2)\\ j = b\\ v \equiv (i, j)}} \alpha(E(v))\;+ \\
    &+\; [\alpha(N(u)) + \alpha(E(u)) + \alpha(S(u)) + \alpha(W(u))],
\end{align*}
and therefore $\alpha(N(u)) + \alpha(E(u)) + \alpha(S(u)) + \alpha(W(u))$ is even, as needed. (Fig.\ \ref{f4}) \qedhere

\end{proof}

\begin{figure}[ht]
\hspace{0pt}\hfill\begin{minipage}{0.45\linewidth}
\centering
\includegraphics[scale=0.625]{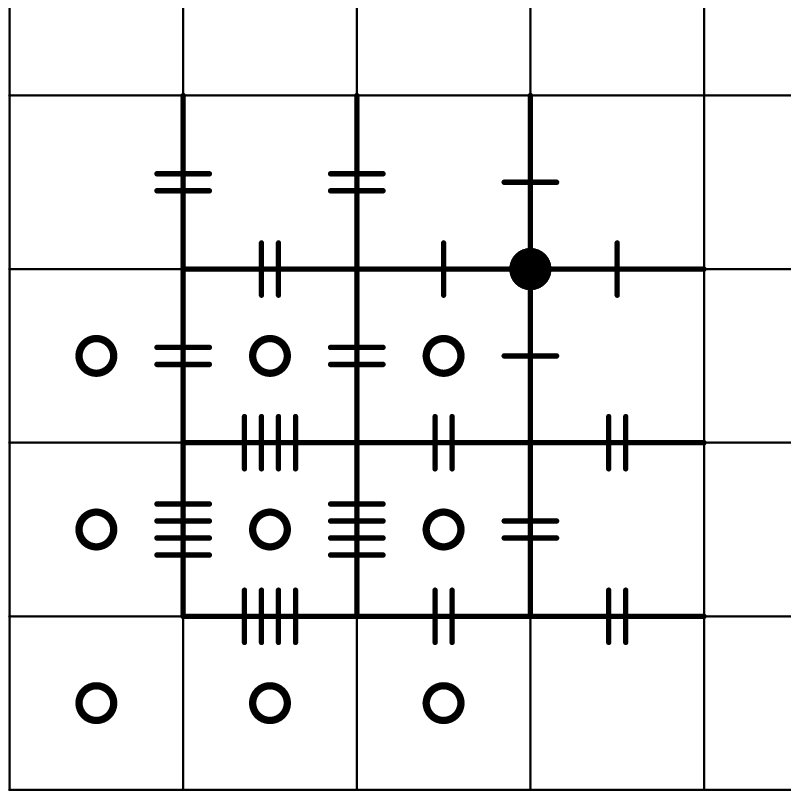}
\caption{}
\label{f4}
\end{minipage}\hfill\begin{minipage}{0.45\linewidth}
\centering
\includegraphics[scale=1.0]{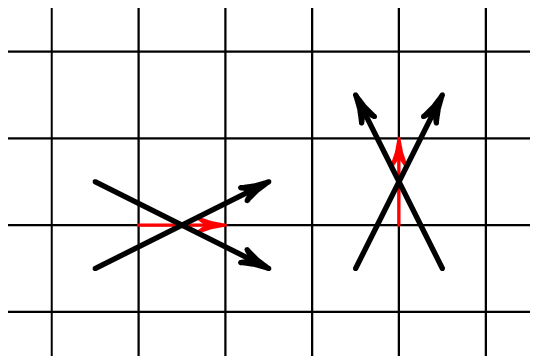}
\caption{}
\label{f5}
\end{minipage}\hfill\hspace{0pt}
\end{figure}

\begin{lemma}
\label{l2}

Every board vertex in $H$ is of degree either 0 or 2.

\end{lemma}

\begin{proof}

Indeed, suppose that the board vertex $v \equiv (i, j)$ in $H$ was of degree 4, the four edges $N(v)$, $E(v)$, $S(v)$, $W(v)$ all being red. Consider the board vertex $u \equiv (i, j + 1)$. (Since the board edge $N(v)$ exists, $u$ must also exist.)

The board edge $W(u)$ cannot be red, because in this case the board edges $W(u)$, $W(v)$, and $S(v)$ being red would imply that the square $C^{-1}(v)$ is connected in $G$ to three other squares, which is a contradiction. The board edge $E(u)$ cannot be red either, because in this case the board edges $E(u)$, $E(v)$, and $S(v)$ being red would imply that the square $D^{-1}(v)$ is connected in $G$ to three other squares, a contradiction. Finally, the board edge $N(u)$, if it exists, cannot be red, because in this case each one of the two squares $C^{-1}(v)$ and $D^{-1}(v)$ would have to be connected to three other squares in $G$, a contradiction.

It follows, then, that $u$ is of degree 1 in $H$, a contradiction with Lemma \ref{l1}. \qedhere

\end{proof}

By Lemma \ref{l2}, $H$ consists of a number of independent simple cycles. In particular, this means that $G$ must necessarily exhibit a braid-like structure such as the one seen in Fig.\ \ref{f3}, with every cycle in $H$ giving rise to a closed braid of crosses.

Orient, now, all cycles in $H$ in an arbitrary manner, and also orient every graph edge in $G$ so that its projection onto its corresponding red board edge $e$ in $H$ points in the same direction as $e$. (Fig.\ \ref{f5})

\begin{lemma}
\label{l3}

Let $u$ and $v$ be two board vertices in $H$ such that there exists a directed simple path $p$ of length $l$ in $H$ from $u$ to $v$. Then there exists a permutation $\sigma$ of the four squares $A^{-1}(v)$, $B^{-1}(v)$, $C^{-1}(v)$, and $D^{-1}(v)$ such that there exist four directed simple paths in $G$ from $A^{-1}(u)$, $B^{-1}(u)$, $C^{-1}(u)$, and $D^{-1}(u)$ to $\sigma(A^{-1}(v))$, $\sigma(B^{-1}(v))$, $\sigma(C^{-1}(v))$, and $\sigma(D^{-1}(v))$, respectively. Moreover, these four paths have no graph edges in common, the set of their graph edges coincides with the set of all graph edges corresponding to a board edge in $p$, and the permutation $\sigma$ has the same parity as $l$.

\end{lemma}

\begin{proof}

By induction on $l$.

If $l = 0$, then $\sigma$ is the identity permutation and all the paths in question are of zero length.

If $l \ge 1$, then let $e$ be the first board edge in $p$, connecting $u$ to $w$. By hypothesis, there exists a permutation $\tau$ of the same parity as $l - 1$ of the squares $A^{-1}(v)$, $B^{-1}(v)$, $C^{-1}(v)$, and $D^{-1}(v)$ such that there exist directed paths $p_a$, $p_b$, $p_c$, and $p_d$ in $G$ from $A^{-1}(w)$, $B^{-1}(w)$, $C^{-1}(w)$, and $D^{-1}(w)$ to $\tau(A^{-1}(v))$, $\tau(B^{-1}(v))$, $\tau(C^{-1}(v))$, and $\tau(D^{-1}(v))$, respectively, which have no graph edges in common and the set of whose graph edges coincides with the set of all graph edges corresponding to a board edge in $p$ other than $e$.

Suppose that $e = E(u)$; all other cases ($e = N(u)$, $e = W(u)$, and $e = S(u)$) are treated analogously.

Let $\pi$ be the permutation \[
\pi = \left( \begin{array}{cccc} A & B & C & D\\ C & A & D & B \end{array} \right),
\] and let $\sigma = \tau \circ \pi$.

Since $B^{-1}(u) \equiv A^{-1}(w)$, the path $p_a$ is also a directed path in $G$ from $B^{-1}(u)$ to $\tau(A^{-1}(v)) \equiv \tau(\pi(B^{-1}(v))) \equiv \sigma(B^{-1}(v))$. Similarly, $p_d$ is also a directed path in $G$ from $C^{-1}(u)$ to $\sigma(C^{-1}(v))$. Finally, splicing together the graph edge leading from $A^{-1}(u)$ to $C^{-1}(w)$ and the path $p_c$ results in a directed path $p'_c$ in $G$ from $A^{-1}(u)$ to $\sigma(A^{-1}(v))$, and splicing together the graph edge leading from $D^{-1}(u)$ to $B^{-1}(w)$ and the path $p_b$ results in a directed path $p'_b$ in $G$ from $D^{-1}(u)$ to $\sigma(D^{-1}(v))$.

It is easy to see that the four paths $p_a$, $p'_b$, $p'_c$, $p_d$ thus formed have no graph edges in common and that the set of their graph edges coincides with the set of all graph edges corresponding to a board edge in $p$. Furthermore, since $\pi$ is an odd permutation and $\tau$ and $l - 1$ have the same parity, $\sigma = \tau \circ \pi$ and $l$ also have the same parity. \qedhere

\end{proof}

\begin{theorem}
\label{t1}

$G$ consists of an even number of cycles.

\end{theorem}

\begin{proof}

Let $s$ be a square in $B$. Since $s$ is connected in $G$ to two other squares, at least one of the board vertices $A(s)$, $B(s)$, $C(s)$, and $D(s)$ is adjacent to a red board edge in $H$. Let $v$ be one such vertex.

Consider the cycle $c$ in $H$ which contains $v$ as a directed path from $v$ to $v$. By Lemma \ref{l3}, there exists a permutation $\sigma$ of $A^{-1}(v)$, $B^{-1}(v)$, $C^{-1}(v)$, and $D^{-1}(v)$ such that there exist directed paths $p_a$, $p_b$, $p_c$, and $p_d$ in $G$ from $A^{-1}(v)$, $B^{-1}(v)$, $C^{-1}(v)$, and $D^{-1}(v)$ to $\sigma(A^{-1}(v))$, $\sigma(B^{-1}(v))$, $\sigma(C^{-1}(v))$, and $\sigma(D^{-1}(v))$, respectively.

Since $c$ contains as many board edges pointing up as ones pointing down, and as many board edges pointing to the left as ones pointing to the right, its total lengths is an even number. By Lemma \ref{l3}, $\sigma$ is then an even permutation. Therefore, when spliced together at the points $A^{-1}(v)$, $B^{-1}(v)$, $C^{-1}(v)$, and $D^{-1}(v)$, the four directed paths $p_a$, $p_b$, $p_c$, and $p_d$ form an even number of directed simple cycles in $G$. Moreover, the set of the graph edges of those cycles coincides with the set of all graph edges corresponding to a board edge in $c$.

And so, for every cycle $C$ in $G$, there exists a cycle $c$ in $H$ which contains all board edges corresponding to a graph edge in $C$; and, for every cycle $c$ in $H$, there exist an even number of cycles in $G$ whose graph edges correspond to board edges in $c$. Summing over all cycles in $H$, it follows that $G$ consists of an even number of cycles, as needed. \qedhere

\end{proof}

\section{Crosspatch Tours}

\begin{theorem}
\label{t2}

No closed crosspatch knight tour exists on a rectangular board.

\end{theorem}

\begin{proof}

This is an immediate corollary of Theorem \ref{t1}. \qedhere

\end{proof}

\begin{theorem}
\label{t3}

No open crosspatch knight tour exists on a rectangular board.

\end{theorem}

\begin{proof}

Suppose that a crosspatch open knight tour $G$ existed on some rectangular board $B$ of size $m \times n$. For every board vertex $v \equiv (a, b)$, partition $B$'s squares into four quadrants as follows: \[
\begin{aligned}
& \{ (i, j) \mid 1 \le i \le a, &1 \le j \le b \},&\\
& \{ (i, j) \mid a < i \le m, &1 \le j \le b \},&\\
& \{ (i, j) \mid a < i \le m, &b < j \le n \},&\;\textrm{and}\\
& \{ (i, j) \mid 1 \le i \le a, &b <  j \le n \}.&
\end{aligned}
\]

There exists a board vertex $u$ such that one of its corresponding quadrants contains exactly one square of degree 1 in $G$. Following the same steps as in the proof of Lemma \ref{l1}, we see that $u$ must have an odd degree in $H$.

Since there are only two endpoint squares in $G$, though, another one of $u$'s respective quadrants must contain no endpoint squares at all. Following the same steps as in the proof of Lemma \ref{l1} once more, we see that $u$ must be of even degree in $H$, a contradiction. \qedhere

\end{proof}

\end{document}